\documentclass[11pt,reqno,a4paper]{amsart}
\usepackage[margin=3cm]{geometry}
\pdfoutput=1

\usepackage{xcolor}
\usepackage[pdfencoding=unicode, hidelinks]{hyperref}
\hypersetup{
	colorlinks,
    breaklinks=true,
	linkcolor={red},
	citecolor={blue},
}
\usepackage{etoolbox}
\cslet{blx@noerroretextools}\empty 
\usepackage{mathtools}
\usepackage{mathrsfs}
\usepackage[symbol]{footmisc}
\usepackage[shortlabels]{enumitem}
\usepackage{microtype}
\usepackage[english]{babel}

\renewcommand{\r}{\mathbb{R}}
\newcommand{\R}{\mathbb{R}}
\newcommand{\C}{\mathbb{C}}
\newcommand{\N}{\mathbb{N}}

\newcommand{\prob}{\mathcal{P}}
\newcommand{\sph}{\mathbb{S}}
\newcommand{\Ws}{W_{\!s}}

\newcommand{\resmes}{\mathbin{\vrule height 1.6ex depth 0pt width 0.13ex\vrule height 0.13ex depth 0pt width 1.3ex}}

\DeclareMathOperator{\supp}{supp}

\DeclareMathOperator{\SO}{SO}

\DeclareMathOperator{\iso}{iso}

\DeclareMathOperator{\loc}{loc}

\numberwithin{equation}{section}

\newtheorem{teo}{Theorem}[section]
\newtheorem{cor}[teo]{Corollary}
\newtheorem{lemma}[teo]{Lemma}

\theoremstyle{definition}

\newtheorem{remark}[teo]{Remark}

\title[Nonlocal anisotropic Riesz interactions with a physical confinement]{Nonlocal anisotropic Riesz interactions\\ with a physical confinement}

\author[M.G.~Mora]{M.G.~Mora}
\author[L.~Rondi]{L.~Rondi}
\author[L.~Scardia]{L.~Scardia}
\author[E.G.~Tolotti]{E.G.~Tolotti}

\address[M.G.~Mora]{Dipartimento di Matematica, Universit\`a di Pavia, Italy}\email{mariagiovanna.mora@unipv.it}

\address[L.~Rondi]{Dipartimento di Matematica, Universit\`a di Pavia, Italy}\email{luca.rondi@unipv.it} 

\address[L.~Scardia]{Department of Mathematics, Heriot-Watt University, United Kingdom}\email{L.Scardia@hw.ac.uk} 

\address[E.G.~Tolotti]{Dipartimento di Matematica, Universit\`a di Pavia, Italy}\email{edoardogiovann.tolotti01@universitadipavia.it}

\begin{document}

    \begin{abstract}
        In this work we fully characterize, in any space dimension, the minimizer of a class of nonlocal and anisotropic Riesz energies defined over probability measures supported on ellipsoids.
        In the super-Coulombic and Coulombic regime, we prove that the minimizer is independent of the anisotropy.
        In contrast, in the sub-Coulombic regime we show that this property fails: we exhibit an example of anisotropy for which the isotropic minimizer is not optimal.
        In order to prove our main result, we provide a formula for the potential inside an ellipsoid, valid in any space dimension and involving the hypergeometric function.
        \bigskip


        \medskip

    \end{abstract}

    \maketitle

    \section{Introduction}

    Nonlocal energies are ubiquitous: they are the continuum counterpart of discrete, pairwise, long-range interactions, and are used to model aggregation phenomena for large numbers of agents of different nature and size, from animals to vehicles, charges, and defects in materials.
    A vast literature is focused on attractive-repulsive energies, where repulsion dominates at small distances, and attraction dominates at large distances.
    Depending on the mutual strength of the attraction and the repulsion, energy minimizers exhibit a variety of patterns, from concentration on points and lower dimensional sets to diffusion.

    In this work we characterize the minimizers of a class of attractive-repulsive energies where the repulsion is driven by an anisotropic Riesz potential, and the attraction is forced by the constraint that minimizers must be compactly supported.
    More precisely, let $\prob(\R^d)$ be the space of probability measures on $\R^d$, $d\geq 2$.
    We consider the repulsive interaction energy 
    \begin{equation}\nonumber
        I_s(\mu) = \int_{\R^d} \int_{\R^d} \Ws(x-y) \, d\mu(x) d\mu(y) 
    \end{equation}
    for $\mu \in \prob(\R^d)$,
    where the interaction kernel $\Ws$, for $s \in (0, d)$, is of the form
    \begin{equation}\label{eq:definition_W}
        \Ws(x) = \frac{1}{|x|^s} \Phi\bigg(\frac{x}{|x|}\bigg)
    \end{equation}
    for $x\neq 0$, and $W_s(0)=+\infty$, and the profile  $\Phi: \sph^{d-1}\to\R$ is a continuous, even, and strictly positive function.

    The attraction is modelled in terms of the indicator function of a compact set; namely given a compact set $E\subset\R^d$, we define the confinement potential
    \begin{equation}\label{eq:definition_VE}
        V_E(x) = \begin{cases}
            0 & \text{if } x \in E,\\
            +\infty & \text{otherwise}.
        \end{cases}
    \end{equation}
    The attractive-repulsive energy we study is then 
    \begin{equation}\label{energ-ysE}
        I_s^{E}(\mu) = I_s(\mu) + \int_{\R^d} V_E(x) \, d\mu (x) 
    \end{equation}
    for $\mu\in\prob(\R^d)$. 
    Clearly, minimizing $I_s^{E}$ over $\prob(\R^d)$ is equivalent to minimizing $I_s$ over the set of measures $\mu\in \prob(\R^d)$ with $\supp\mu \subset E$.

    The special choices $\Phi\equiv1$ in \eqref{eq:definition_W} and $E=B_1$ in \eqref{eq:definition_VE}, where $B_1$ denotes the closed unit ball of $\R^d$ centered at the origin, correspond to a radially symmetric energy $I_{\iso,s}^{B_{1}}$, which has been widely studied in the last century. It is well known that  the unique minimizer of $I_{\iso,s}^{B_{1}}$ is the measure $\mu_{\iso, s}$
    defined as
    \begin{equation}\nonumber
        \mu_{\iso, s} = \begin{cases}
            c_{s, d}(1-|x|^{2})^{\frac{s-d}{2}} \mathcal{L}^{d}(x) \resmes B_{1} \quad & \text{if } d-2< s <d ,\smallskip\\
            c_{s, d} \mathcal{H}^{d-1} \resmes \partial B_{1} \quad & \text{if } 0<s \leq d-2,
            \end{cases}
    \end{equation}
    where $c_{s, d}$ is a normalization constant (see \eqref{cuq:def} below). The minimality of $\mu_{\iso, s}$ for $I_{\iso,s}^{B_{1}}$ dates back to Riesz \cite{RIESZ1988-2,RIESZ1988}.
    The cases $d = 1, 2, 3$ had already been solved a few years earlier by P\'olya and Szeg\H{o} \cite{POLYA1931}. 
    More recently there has been a renewed interest in this result, and alternative methods of proofs have been provided, see e.g.\ \cite{CHAFAI2022,Dyda}. 

    We observe that the Coulomb exponent $s=d-2$ acts as a threshold for the dimensionality of the minimizer. At the threshold value $s=d-2$ the minimizer concentrates on the boundary of the constraint. If the potential is more singular at the origin, then the minimizing measure gets more and more diffuse, with a density depending on the Riesz exponent $s>d-2$. On the other hand, if the potential is less singular than Coulombic, then $\mu_{\iso,s}=\mu_{\iso,d-2}$ for every $0<s\leq d-2$. In other words, once the minimizer concentrates at the Coulomb threshold $s=d-2$, it remains the same for all the sub-Coulombic Riesz exponents. 

    In this work we mainly focus on super-Coulombic Riesz interactions $s\geq d-2$, and confining sets $E$ given by ellipsoids. 
    Our main result is the following.

    \begin{teo}
        \label{thm:main}
        Let $d\geq 2$, $s\in(0,d)$ with $s\geq d-2$, and let $\Ws$ be as in \eqref{eq:definition_W} with $\Phi:{\mathbb S}^{d-1}\to\R$ a continuous, even, and strictly positive function. Assume that $\widehat \Ws$
        is continuous and non-negative on $\sph^{d-1}$. 
        Let $E\subset \R^d$ be an ellipsoid of the form $E=RDB_1$ for some $R\in \SO(d)$ and some positive-definite $d\times d$ diagonal matrix $D$, and let $T^{E}$ be the map $T^E(x) = RDx$ for $x\in\R^d$.
        Then the unique minimizer of $I_s^{E}$ over $\prob(\R^d)$ is the push-forward of the measure $\mu_{\iso, s}$ by the map $T^{E}$, that is, the probability measure
        \begin{equation}\nonumber
        \mu_{s}^E = \begin{cases}
                |E|^{-1}\dfrac{\Gamma(1+\frac{s}{2})}{\Gamma(1+\frac d2)\Gamma(1+\frac{s-d}{2})}(1-|D^{-1}R^Tx|^{2})^{\frac{s-d}{2}} \mathcal{L}^{d}(x) \resmes E \quad & \text{if } d-2< s <d ,\medskip\\
                 \big(\mathcal{H}^{d-1}(\partial B_1)\det D\big)^{-1}|D^{-2}R^Tx|^{-1}  \mathcal{H}^{d-1}(x) \resmes \partial E\quad & \text{if } 0< s=d-2.
                \end{cases}
        \end{equation}
    \end{teo}

    In the statement above we require $s>0$, hence logarithmic interactions in two dimensions, which correspond loosely speaking to $s=0$, are not included in our analysis. However,
    the results in \cite{Mora-Muenster, MS} show that also in this case minimizers are the push-forward of the measure $(2\pi)^{-1}\mathcal{H}^1\resmes\partial B_1$ by $T^E$, in agreement with the statement of Theorem~\ref{thm:main} for the Coulomb scaling $s=d-2$.

    Surprisingly, the minimizing measure $\mu_s^E$ is completely independent of the profile $\Phi$, and its support is fully determined by the confinement term. 
    It is natural to ask whether this phenomenon occurs also in the sub-Coulombic regime $s<d-2$. In Section~\ref{Sec:counter} we give a negative answer to this question:
    we consider $E=B_1$ and we show that, for a suitable profile $\Phi$, the measure $c_{s, d} \mathcal{H}^{d-1} \resmes \partial B_{1}$ does not satisfy the Euler--Lagrange conditions for $I_s^{B_1}$ when $s<d-2$. 
    Hence, for $0<s<d-2$ the anisotropy $\Phi$ \textit{does} play a role in determining the energy minimizers, unlike for $s\geq d-2$.

    The role of the anisotropy in determining the shape and dimension of the  minimizer has been the subject of recent attention in the case of quadratic attraction, namely for attractive-repulsive energies as in \eqref{energ-ysE}, with $V_E(x)$ replaced by $|x|^2$. The most general result to date is \cite{RieszFMMRSV}, where the authors extend the papers \cite{CARRILLO2023,CARRILLO2023-2,MMRSV2023}
    beyond dimensions two and three. Unlike the case of physical confinement treated here, however, the characterization of the minimizer in \cite{RieszFMMRSV} has been proved under some restrictions on the Riesz exponent $s$, and consequently on the space dimension $d$, which  are stricter than for the isotropic counterpart $\Phi\equiv1$. 
    We refer to \cite{FRANKMatzke2023} for a complete overview of the available results in the isotropic case.

    For the proof of Theorem~\ref{thm:main} we show that, for any ellipsoid $E$ and any $s\geq d-2$, the corresponding measure $\mu_s^E$ satisfies the Euler--Lagrange conditions for the energy $I_s^E$ (see (EL1)--(EL3) in Theorem~\ref{thm:exun}). To this aim, the strategy devised in \cite{MMRSV2023}, and further developed in \cite{RieszFMMRSV}, requires rewriting the potential $\Ws\ast\mu_s^E$, for a general ellipsoid $E$, in Fourier form. Due to the lack of integrability of the Fourier transform of $\Ws\ast\mu_s^E$ at infinity, however, this can be done only for the regularized potentials $P_r=\frac{\chi_{B_r}}{|B_r|}\ast\Ws\ast\mu_s^E$, for $r>0$. The representation of $\Ws\ast\mu_s^E$ is then obtained in Theorem~\ref{teo:formula_convolution} via a careful limit procedure as $r\to 0^+$.

    We note that Theorem~\ref{teo:formula_convolution} provides a new general formula for the potential $\Ws\ast\mu_q^E$ inside $E$, valid for all $0<s<d$ and $q\geq d-2$. In particular, this extends the formula derived in \cite{RieszFMMRSV} to the full range of Riesz exponents in any space dimension. A Fourier representation of the potential $\Ws\ast\mu_q^E$ outside $E$ is unfortunately lacking for general $s$ and $q$. However, in the present setting of a physical confinement,
  only the expression of the potential inside $E$ is needed. For this reason the formula provided by Theorem~\ref{teo:formula_convolution} is sufficient to establish our result without any restrictions on the space dimension $d$ or the Riesz exponent $s\geq d-2$, unlike the case of quadratic attraction considered in \cite{RieszFMMRSV}, which remains open in high dimensions and for large Riesz exponents.
  
    The paper is structured as follows. In Section~\ref{sect-prelim} we collect some definitions and preliminary results that will be used throughout the paper. Section~\ref{sect:superC} is devoted to the proof of Theorem~\ref{thm:main}, where we obtain a characterization of the minimizer for $s\geq d-2$. The case $s<d-2$ is discussed in Section~\ref{Sec:counter}, where we provide an explicit example illustrating the role of anisotropy in the sub-Coulombic regime.

    \section{Preliminary results}\label{sect-prelim}
    In this section we collect some definitions and results that will be needed in the paper.

        \subsection{Some special functions}\label{sect:special}

        We recall the definition and the main properties of some special functions that will be used in the rest of the manuscript.
        For further details we refer to \cite{ERDELYI1953,LEBEDEV1965}. 

            \subsubsection*{The Gamma and the Beta functions}
            The Gamma function is defined as 
            \begin{equation}\nonumber
                \Gamma(z) = \int_{0}^{\infty} e^{-t}t^{z-1} \, dt 
            \end{equation}
            for $z\in\C$ with $\Re(z) > 0$.
            It can be extended by analytic continuation to the whole complex plane, except at non-positive integers.
            We will use the following properties of $\Gamma$:
            \begin{enumerate}[(i)]
                \item $\Gamma(z + 1) = z \Gamma(z)$ 
                for every $z\in\C$, $z\neq 0,-1,-2,\dots$ (fundamental property),\smallskip
                \item $\Gamma(z)\Gamma\left(z + \frac{1}{2}\right) = 2^{1-2z}\sqrt{\pi} \Gamma(2z)$ for every $z\in\C$, $z\neq 0,-\frac12, -1,-\frac32,\dots$ (Legendre duplication formula),\smallskip
                \item $\Gamma\left(\frac{1}{2}\right) = \sqrt{\pi}$.
            \end{enumerate}

            The Beta function can be defined in terms of the Gamma function as 
            \begin{equation}\label{Beta0}
                B(x, y) =\frac{\Gamma(x)\Gamma(y)}{\Gamma(x+y)} 
            \end{equation}
            for $x,y\in\C$ with $\Re(x)>0$ and $\Re(y)>0$.
            It easily follows that $B$ is symmetric. Moreover, for $\mu , \nu > 0$,
             \begin{equation}\label{Beta}
                    \int_{0}^{\frac{\pi }{2}} \sin^{\mu -1}(t)\cos^{\nu -1}(t) \, dt = \frac{1}{2} B\left(\frac{\mu }{2}, \frac{\nu }{2}\right),
                \end{equation}
            see \cite[Formula 3.621--5]{GRADSHTEYN2007}.

            \subsubsection*{The hypergeometric function}
            Let $\alpha,\beta,\gamma \in\C$ with $\gamma\neq0,-1,-2,\dots$.
            The hypergeometric function ${}_{2}F_{1}$ is defined as the power series
            \begin{equation}\nonumber
                {}_{2}F_{1}(\alpha, \beta; \gamma; z) = \sum_{n=0}^{\infty} \frac{(\alpha)_{n}(\beta)_{n}}{(\gamma)_{n}n!} z^{n}
            \end{equation}
            for $z\in \C$ with $|z|< 1$.
            Here $(\lambda)_{n}$ denotes the Pochhammer's symbol, namely
            \begin{equation}\nonumber
                (\lambda)_{0} = 1 , \qquad (\lambda)_{n} = \lambda(\lambda+1)\cdot\dots\cdot(\lambda+n-1) \quad \text{ for } n\in\N.
            \end{equation}
            If $-1<\Re(\gamma-\alpha-\beta)\leq 0$, then the series converges for $|z|\leq1$, except at the point $z=1$. More precisely,  
            we have the following behavior of the series at $z=1$.
            If $\gamma=\alpha+\beta$, 
            \begin{equation}\label{limz2} 
            \lim_{z\to 1^-}\frac{{}_2F_1(\alpha,\beta;\alpha+\beta;z)}{-\log(1-z)}=\frac{\Gamma(\alpha+\beta)}{\Gamma(\alpha)\Gamma(\beta)},
            \end{equation}
            while if $\Re(\gamma-\alpha-\beta)<0$,
            \begin{equation}\label{limz3}
            \lim_{z\to 1^-}\frac{{}_2F_1(\alpha,\beta;\gamma;z)}{(1-z)^{\gamma-\alpha-\beta}}=\frac{\Gamma(\gamma)\Gamma(\alpha+\beta-\gamma)}{\Gamma(\alpha)\Gamma(\beta)}.
            \end{equation}
            If instead $\Re(\gamma-\alpha-\beta)>0$, the series extends continuously also at $z=1$, and 
            \begin{equation}\label{limz1}
            \lim_{z\to 1^-}{}_2F_1(\alpha,\beta;\gamma;z)={}_2F_1(\alpha,\beta;\gamma;1)=\frac{\Gamma(\gamma)\Gamma(\gamma-\alpha-\beta)}{\Gamma(\gamma-\alpha)\Gamma(\gamma-\beta)}.
            \end{equation}
            For more details we refer to \cite[Chapter 9]{LEBEDEV1965}.

            Note that if $\alpha$ or $\beta$ is a non-positive integer, then the hypergeometric function reduces to a polynomial in $z$.
            In particular, it holds that
            \begin{equation}\label{2F10}
                {}_{2}F_{1}(0, \beta; \gamma; z) = 1
            \end{equation}
            and
            \begin{equation}\label{2F1-1}
                {}_{2}F_{1}(-1, \beta; \gamma; z) = 1-\frac{\beta}{\gamma}z.
            \end{equation}

            \subsubsection*{The Appell function of the fourth kind}
            Let $\alpha,\beta,\gamma,\gamma' \in\C$ with $\gamma,\gamma'\neq0,-1,-2,\dots$.
            The Appell function of the fourth kind $F_{4}$ is defined as the double power series
            \begin{equation}\nonumber
                F_{4}(\alpha, \beta; \gamma, \gamma'; x, y) = \sum_{n, m = 0}^{\infty} \frac{(\alpha)_{m + n}(\beta)_{m+n}}{(\gamma)_{m}(\gamma')_{n}n!m!} x^{m} y^{n}
            \end{equation}
            for $x,y\in\C$ with $\sqrt{|x|} + \sqrt{|y|} < 1$.
            Since $F_{4}$ is analytic in its domain, it follows that
            \begin{equation}\label{eq:limit_appel}
                \lim_{x \to 0} F_{4}(\alpha, \beta; \gamma, \gamma';x, y) = {}_{2} F_{1}(\alpha, \beta; \gamma'; y) 
            \end{equation}
            for every $y\in\C$ with $|y| < 1$.

            \subsubsection*{The Bessel function of the first kind}
            Let $\nu\geq0$.
            The Bessel function of the first kind of order $\nu$ is defined as 
            \begin{equation}\nonumber
                J_{\nu}(x) = \sum_{n=0}^{\infty} \frac{(-1)^{n}(x / 2)^{\nu+2n}}{\Gamma(n+1)\Gamma(n + \nu + 1)} 
            \end{equation}
            for $x \geq 0$.
            We recall the asymptotic behavior of $J_\nu$ at $0$ and $+\infty$:
            \begin{align}
                J_{\nu}(x) & \sim \frac{x^{\nu}}{2^{\nu}\Gamma(1+\nu)} , && \text{as } x \to 0^{+}, \label{eq:bessel_zero}\\
                J_{\nu}(x) & \sim \sqrt{\frac{2}{\pi x}} \cos\left(x-\frac{\pi}{2}\nu - \frac{\pi}{4}\right) , && \text{as } x \to+\infty, \label{eq:bessel_infinity}
            \end{align}
            see \cite[Section~5.16]{LEBEDEV1965}.

        \subsection{The Fourier transform}
        Let $\mathcal{S}(\R^d)$ denote the Schwartz class of rapidly decreasing functions and let $\mathcal{S}'(\R^d)$ be its dual space, that is, the space of tempered distributions. Following \cite{FOLLAND1992},
        we define the Fourier transform $\widehat f$ of a function $f \in \mathcal{S}(\R^d)$ as
        \begin{equation}\nonumber
            \widehat f (\xi) = \int_{\R^d} f(x)e^{- i x \cdot\xi} \, dx 
        \end{equation}
        for $\xi\in\R^d$. If $f$ is a radial function, that is, $f(x)=f_0(|x|)$ for some $f_0$, then so is $\widehat f$ and we have
        \begin{equation}\label{radFourier}
        \widehat{f}(\xi)=
        \frac{(2\pi)^{\frac{d}{2}}}{|\xi|^{\frac{d}{2}-1}} \int_{0}^{\infty} f_0(r) J_{\frac{d}{2}-1}(r|\xi|) r^{\frac{d}{2}} \, dr,
        \end{equation}
        see \cite[eq.~(7.38)]{FOLLAND1992}.

        If $f$ is a tempered distribution, its Fourier transform can be defined by duality, namely
        \begin{equation}\nonumber
            \langle \widehat f, \varphi \rangle = \langle f, \widehat \varphi \rangle
        \end{equation}
        for every $\varphi\in \mathcal{S}(\R^d)$. 
        By the inversion theorem for the Fourier transform, if $f$ is a tempered distribution such that $f\in L^{1}(\R^d)$ and $\widehat f \in L^{1}(\R^d)$, then
        \begin{equation}\nonumber
            f(x) = \frac{1}{(2\pi)^{d}} \int_{\R^d} \widehat f(\xi) e^{ix\cdot \xi} \, d\xi
        \end{equation}
        for a.e.\ $x\in\R^d$,
        see for example \cite[p.~244]{FOLLAND1992}.

        Let $\mu$ be a Radon measure in $\R^d$ with compact support. It is easy to see that $\mu \in \mathcal{S}'(\R^d)$
        and $\widehat \mu \in C^{0}(\R^d)$.
        Moreover, if $f \in \mathcal{S}'(\R^d)$, the convolution $f \ast \mu$ is well-defined as an element of $\mathcal{S}'(\R^d)$ 
        and $\widehat{f \ast \mu} = \widehat f \, \widehat \mu$, see \cite[Theorem~1.7.6]{Hor}.

        \subsection{The Fourier transform of $\Ws$}\label{Fourier-Ws}
        Let $\Ws$ be a kernel as in Theorem~\ref{thm:main}. Since $\Ws$ is locally integrable and with sublinear growth at infinity, it is a tempered distribution.
        To compute the Fourier transform of $\Ws$, it is convenient to write the profile $\Phi\in L^2(\sph^{d-1})$ in terms of spherical harmonics, namely
        $$\Phi=\sum_{n=0}^{\infty}\phi_n,$$
        where each $\phi_n$ is a spherical harmonic of order $n$ on $\sph^{d-1}$.
        Then for $x\in \R^d$
        $$\Ws(x)= \sum_{n=0}^{\infty}\frac{1}{|x|^s}\,\phi_n\bigg(\frac{x}{|x|}\bigg).$$ 
        By \cite[Chapter V, Lemma~2]{Stein} for $k=0$ and \cite[Chapter III, Theorem~5]{Stein} for $k\geq 1$, the Fourier transform of $p_{k}(x)|x|^{-s-k}$ with $p_k$
            a homogeneous harmonic polynomial of degree $k$ is given by    
            \begin{equation}\label{Stein-F}
                (-i)^{k}2^{d-s}\pi^{\frac{d}{2}}\frac{\Gamma (\frac{k+d-s}{2})}{\Gamma (\frac{k+s}{2})} \frac{p_{k}(\xi)}{|\xi|^{k+d-s}}.  
            \end{equation}
        We infer that for $\xi\in \R^d$, for suitable constants $b_{n,s,d}$,
        \begin{equation}\label{hatW}
        \widehat{\Ws}(\xi)=\frac{1}{|\xi|^{d-s}}\sum_{n=0}^{\infty}b_{n,s,d} \,\phi_n(\xi),
        \end{equation}
        provided the series at the right-hand side converges in $L^2(\sph^{d-1})$ to a function which we denote by $\Psi:=\widehat{\Ws}|_{\sph^{d-1}}$.
        In Theorem~\ref{thm:main} such a function $\Psi$ is assumed to be continuous on $\sph^{d-1}$.
        Note that, since $\Phi$ is even, $\Psi$ is even and real-valued.

        \subsection{Existence and uniqueness of the minimizer}
        The following result is by now standard.
        \begin{teo}\label{thm:exun}
        Let $\Ws$ be a kernel of the form \eqref{eq:definition_W} with $s\in(0,d)$ and let $\Phi:{\mathbb S}^{d-1}\to\R$ be a continuous, even, and strictly positive function. Assume that $\widehat \Ws$
        is continuous and non-negative. 
        Let $E\subset \R^d$ be a compact set of positive $(d-s)$-capacity.
        Then
        the functional $I_{s}^{E}$ has a unique minimizer $\mu$ over $\mathcal{P}(\R^d)$. Moreover, $\mu$ is the unique measure in $\mathcal{P}(\R^d)$ satisfying the following Euler--Lagrange conditions: there exists a constant $C\in\R$ such that
        \begin{align}
         & \supp\mu \subset  E, \label{eq:third_el_condition}\tag{EL1}\\
         &  (\Ws \ast \mu)(x)  = C \quad \text{ for } \mu  {-a.e.\ } x \in \supp\mu,\label{eq:first_el_condition}\tag{EL2}\\
           & (\Ws \ast \mu)(x)  \geq C \quad \text{ for every }x \in  E \setminus N \text{ with } {\rm Cap}_{d-s}(N) = 0.\label{eq:second_el_condition}\tag{EL3}
            \end{align}
        \end{teo}

        We refer to \cite{LANDKOF1972} for the notion of $(d-s)$-capacity.

        \begin{proof}[Proof of Theorem~\ref{thm:exun}] The proof is an adaptation of \cite[Proposition~2.1]{RieszFMMRSV}. Note that the tightness of minimizing sequences, as well as the compactness of the support of the minimizer, are immediate consequences of the presence of the confinement potential $V_E$ in the functional $I_s^E$. 
        \end{proof}

    \section{Characterization of the minimizer for $s \geq d-2$}\label{sect:superC}

    In this section we prove Theorem~\ref{thm:main}. 
    A general ellipsoid in $\R^d$ centered at the origin can be described as $E=RDB_1$, where $R \in \SO(d)$, $D$ is a positive-definite $d\times d$ diagonal matrix, and $B_1$ is the closed unit ball centered at the origin.
    Given an ellipsoid $E = RDB_1$, we define the linear map $T^{E}(x) = RDx$ for $x\in\R^d$.

    For any $q\geq d-2$, we define
    \begin{equation}\label{muq:def}
        \mu_{ q} = \begin{cases}
            c_{q, d}(1-|x|^{2})^{\frac{q-d}{2}} \mathcal{L}^{d}(x) \resmes B_{1} \quad & \text{if } q> d-2,\smallskip\\
            c_{q, d} \mathcal{H}^{d-1} \resmes \partial B_{1} \quad & \text{if } q= d-2,
            \end{cases}
    \end{equation}
    where $c_{q,d}$ is a normalization constant so that $\mu_{q}$ is a probability measure. Namely, using for instance \cite[Formula 3.251--1]{GRADSHTEYN2007}, 
    \begin{equation}\label{cuq:def}
        c_{q, d} = \begin{cases} 
            \displaystyle |B_1|^{-1}\frac{\Gamma(1+\frac{q}{2})}{\Gamma(1+\frac d2)\Gamma(1+\frac{q-d}{2})}  =\pi^{-\frac{d}{2}}\frac{\Gamma(1+\frac{q}{2})}{\Gamma(1+\frac{q-d}{2})} \quad & \text{if } q>d-2,\smallskip\\
            \displaystyle  \big(\mathcal{H}^{d-1}(\partial B_1)\big)^{-1} = \pi^{-\frac{d}{2}}\frac{\Gamma(\frac{d}{2})}{2 } \quad & \text{if } q= d-2 .
            \end{cases}
    \end{equation}
     For $s=q\in [d-2,d)\cap (0,d)$, $\mu_q$ clearly coincides with $\mu_{\iso,s}$. The push-forward of $\mu_q$ by the map $T^E$ is given by
      \begin{equation}\nonumber
    \mu_{q}^E = \begin{cases}
            |E|^{-1}\dfrac{\Gamma(1+\frac{q}{2})}{\Gamma(1+\frac d2)\Gamma(1+\frac{q-d}{2})}(1-|D^{-1}R^Tx|^{2})^{\frac{q-d}{2}} \mathcal{L}^{d}(x) \resmes E \quad & \text{if } q>d-2,\medskip\\
            \big(\mathcal{H}^{d-1}(\partial B_1)\det D\big)^{-1}|D^{-2}R^Tx|^{-1}  \mathcal{H}^{d-1}(x) \resmes \partial E \quad & \text{if } q=d-2.
            \end{cases}
    \end{equation}
    We start this section by computing the Fourier transform of the measure $\mu_{q}$ for $q\geq d-2$ and of its push-forward $\mu_{q}^{E}$ by the map $T^E$.
    Note that $\mu_{ q}$ and $\mu_{q}^{E}$ are Radon measures with compact support, hence they are tempered distributions.

    \begin{lemma}\label{lemma:fourier_mus}
        Let $q \geq d-2$, and let $\mu_q\in \mathcal{P}(\R^d)$ be defined as in \eqref{muq:def}. Then
        \begin{equation}\nonumber
            \widehat \mu_{q}(\xi) = \hat c_{q}\frac{1}{|\xi|^{\frac{q}{2}}}J_{\frac{q}{2}}(|\xi|)  ,
        \end{equation}
        where
        \begin{equation}\nonumber
            \hat c_{q} = 2^{\frac{q}{2}} \Gamma\left(1+\frac{q}{2}\right).
        \end{equation}
        Moreover, if $E$ is an ellipsoid of the form $E = RDB_1$ with $R \in \SO(d)$ and $D$ a positive-definite $d\times d$ diagonal matrix, then
        \begin{equation}\label{cor:fourier_muEs}
            \widehat \mu_{q}^{E}(\xi) = \widehat \mu_{ q} (DR^{T}\xi) .
        \end{equation}
    \end{lemma}

    \begin{proof}
        We start with the case $q > d-2$.
        Applying the formula \eqref{radFourier} for the Fourier transform of a radial function  we get
        \begin{equation}\label{eq:radial_fourier_mus}
            \widehat \mu_{ q}(\xi) = c_{q, d} \frac{(2\pi)^{\frac{d}{2}}}{|\xi|^{\frac{d}{2}-1}} \int_{0}^{1} (1 - r^{2})^{\frac{q-d}{2}} J_{\frac{d}{2}-1}(r|\xi|) r^{\frac{d}{2}} \, dr.
        \end{equation}
        To compute the integral in \eqref{eq:radial_fourier_mus} we use the identity
    \begin{equation}\label{eq:formula65671}
            \int_{0}^{1}x^{\nu+1}(1-x^{2})^{\rho}J_{\nu}(bx)\, dx = \frac{2^{\rho}\Gamma(\rho + 1)}{b^{\rho+1}} J_{\nu+\rho+1}(b),
        \end{equation}
        which holds for $b>0$, $\Re(\nu) > -1$, and $\Re(\rho) > -1$ (see \cite[Formula 6.567--1]{GRADSHTEYN2007}).
           Applying \eqref{eq:formula65671} with $\nu = \frac{d}{2}-1$, $\rho = \frac{q-d}{2}$, and $b = |\xi|$, we then obtain from \eqref{eq:radial_fourier_mus} that
        \begin{align*}
            \widehat \mu_{ q}(\xi)  = c_{q, d} \pi^{\frac{d}{2}}2^{\frac{q}{2}}\Gamma\left(1+\frac{q-d}{2}\right)\frac{1}{|\xi|^{\frac{q}{2}}}J_{\frac{q}{2}}(|\xi|)  = 2^{\frac{q}{2}}\Gamma\left(1+\frac{q}{2}\right) \frac{1}{|\xi|^{\frac{q}{2}}} J_{\frac{q}{2}}(|\xi|),
        \end{align*}
        which proves the first claim in the statement.
        For a similar computation see also \cite[Appendix B.5]{GRAFAKOS2008}, where a slightly different definition of the Fourier transform is used.

        When $q = d-2$, we follow the computation in \cite[Appendix B.4]{GRAFAKOS2008} and obtain
        \begin{align*}
            \widehat \mu_{d-2}(\xi)  = c_{d-2, d} \int_{ \partial B_{1}} e^{-i\xi\cdot \omega} \, d\mathcal{H}^{d-1}(\omega) 
            = 2^{\frac{d}{2}-1} \Gamma\left(\frac{d}{2}\right)\frac{1}{|\xi|^{\frac{d}{2}-1}}J_{\frac{d}{2}-1}(|\xi|).
             \end{align*}
        
        Finally, \eqref{cor:fourier_muEs} follows from the linearity and invertibility of the map $T^{E}$.
    \end{proof}

    The key ingredient for the proof of our main result, Theorem~\ref{thm:main}, is a formula for the expression of the potential function $\Ws \ast \mu^{E}_{q}$
    inside $E$ when $s \in (0, d)$ and $q\in [d-2,+\infty)$.
    As a first step, in the next technical lemma we study the regularity of $\Ws \ast \mu^{E}_{q}$.

    \begin{lemma}\label{regular_lemm}
        Let $s\in(0,d)$ and $q\in[d-2,+\infty)$. Let $\Ws$ be as in \eqref{eq:definition_W}, let $E = RDB_1$ be an ellipsoid, and let $\mu_{q}^{E}$ be the push-forward of the measure $\mu_{q}$ by the map $T^{E}$.
        Then $\Ws \ast \mu_{q}^{E}\in L^1_{\loc}(\mathbb{R}^d)\cap C^0(\R^d\setminus \partial E)$. Moreover, if $0<s<\min(d,\frac{q+d}2)$, then $\Ws \ast \mu_{q}^{E}\in C^0(\R^d)$.
    \end{lemma}

    \begin{proof}
        We only sketch the proof in the case $E=B_1$ and $\Phi\equiv 1$, and for convenience we ignore the normalization constant $c_{q,d}$.
        It is immediate to show that, for any $0<s<d$ and $q > d-2$, we have $\Ws \ast \mu_{q}\in L^1_{\loc}(\mathbb{R}^d)\cap C^0(\R^d\setminus \partial B_1)$. 
        The continuity in $\R^d\setminus \partial B_1$ also holds for $q=d-2$. We focus on the behavior of the potential across $\partial B_1$. 
        To this aim we consider, with no loss of generality, $(\Ws \ast \mu_{q})(t e_1)$ for $t$ close to $1$. 
        We use the compact notation 
        \begin{equation}\label{pot:te1}
            (\Ws \ast \mu_{q})(t e_1)=
            \begin{cases}
                \smallskip 
                v_s(t,1) \quad &\text{if } q=d-2,\\
                \displaystyle\int_0^1v_s({\color{blue}t,r})(1-r^2)^{\frac{q-d}2}\,dr \quad &\text{if } q>d-2,
            \end{cases}
        \end{equation}
        where 
        \[
            v_s(t,r):=\int_{\partial B_r} \frac{1}{|te_1-y|^s}\,d\mathcal{H}^{d-1}(y),
        \]
        for $t,r>0$ and $t\neq r$.
        By setting $y=r\omega$ and using spherical coordinates on $\partial B_1$ we have
        \begin{align*}
        v_s(t,r)&=C(d)r^{d-1}\int_0^\pi
        \frac{\sin^{d-2}(\varphi)}{(t^2+r^2-2rt\cos(\varphi))^{s/2}}\,d\varphi\\
        &=\frac{C(d)r^{d-1}}{h^s}\int_0^\pi
        \frac{\sin^{d-2}(\varphi)}{(1+\alpha^2-2\alpha\cos(\varphi))^{s/2}}\,d\varphi,
        \end{align*}
        where $h=\max(t,r)$, $\alpha=\min(t,r)/h\in {\color{blue}(0,1)}$, and $C(d)>0$ is a dimensional constant. 
        The integral above can be computed explicitly (see \cite[Formula 3.665--2]{GRADSHTEYN2007} with $\mu=\frac{d-1}2$, $\nu=\frac s2$, $a =\alpha$), and gives 
        \begin{equation}\label{vBF}
        v_s(t,r)=\frac{C(d)r^{d-1}}{h^s} B\left(\frac{d-1}2,\frac12\right)\, {}_{2}F_{1}\left(\frac s2,\frac{s-d+2}{2};\frac d2;\alpha^2\right).
        \end{equation}
        We now treat the cases $0<s<d-1$ and $d-1\leq s<d$ separately.

        Let $0<s<d-1$.
        As recalled in Section \ref{sect:special}, ${}_{2}F_{1}$ is continuous with respect to the last variable in the interval $[0,1]$ whenever $\frac d2-\frac s2-\frac{s-d+2}{2}>0$, namely for $0<s<d-1$.
        Hence, by \eqref{vBF} $v_s$ is continuous in $t$.
        For $q= d-2$, this immediately provides the continuity of the potential in the whole $\R^d$. For $q>d-2$ the continuity follows by \eqref{pot:te1}, \eqref{vBF}, and the Dominated Convergence Theorem.
        We note that for $0<s<d-1$ and $q\geq d-2$, we have
    $\min(d,\frac{q+d}2)\geq d-1>s$.

        Let now $d-1<s<d$. In this range of Riesz exponents, by \eqref{limz3}, 
        \begin{equation}\label{asy:F21}
        {}_{2}F_{1}\left(\frac s2,\frac{s-d+2}{2};\frac d2;\alpha^2\right)
        \sim (1-\alpha^2)^{d-s-1} \quad \text{for } \alpha\sim 1.
        \end{equation}
        If $q=d-2$, the asymptotics in \eqref{asy:F21}, with \eqref{pot:te1}--\eqref{vBF} and the fact that $d-s-1>-1$, implies that $\Ws \ast \mu_{q}$ is integrable in a neighborhood of $\partial B_1$, hence it is locally integrable in $\R^d$. 
        The same asymptotics shows that the potential blows up on $\partial B_1$, thus it is not continuous on $\R^d$. Note that for $d-1<s<d$ and $q=d-2$, we have $\min(d,\frac{q+d}2)= d-1<s$.

         Let now $q>d-2$. Let $0<\varepsilon<1/4$, and assume that $|t-1|<\varepsilon$. In what follows we consider the integral
        \begin{equation}\label{claim:t1}
            \int_{1-2\varepsilon}^1v_s({\color{blue}t,r})(1-r^2)^{\frac{q-d}2}\,dr,
        \end{equation}
        which is crucial in proving the continuity of the potential.
 Since $|t-1|<\varepsilon$, $1-2\varepsilon <r <1$, and $t\neq r$, we have $\frac{1-2\varepsilon}{1+\varepsilon}\leq \alpha<1$. Hence, for $\varepsilon$ small enough, we can replace, up to constants, the hypergeometric function in \eqref{vBF} with its asymptotics given by \eqref{asy:F21}. Namely, 
        we can estimate \eqref{claim:t1} as follows:
        \[
            \int_{1-2\varepsilon}^1v_s(t,r)(1-r^2)^{\frac{q-d}2}\,dr \sim \int_{1-2\varepsilon}^1\frac{r^{d-1}}{h^s}(1-\alpha^2)^{d-s-1}(1-r^2)^{\frac{q-d}2}\,dr=:I.
        \]
     We distinguish two cases, depending whether $t$ is smaller or greater than $1$. 
     
     Let $t<1$. We split $I$ into $I=I_1+I_2$, where
        \begin{align*}
            I_1&=\int_{1-2\varepsilon}^t\frac{r^{d-1}}{h^s}(1-\alpha^2)^{d-s-1}(1-r^2)^{\frac{q-d}2}\,dr=\int_{1-2\varepsilon}^t \frac{r^{d-1}} {t^{2d-s-2}}  (t^2-r^2)^{d-s-1}(1-r^2)^{\frac{q-d}2}\,dr, \\ 
            I_2&=\int_t^1\frac{r^{d-1}}{h^s}(1-\alpha^2)^{d-s-1}(1-r^2)^{\frac{q-d}2}\,dr= \int_t^1  r^{s+1-d}(r^2-t^2)^{d-s-1}(1-r^2)^{\frac{q-d}2}\,dr.
        \end{align*}
       To estimate $I_2$ it is convenient to further split $I_2=I_3+I_4$, where 
        \begin{align*}
            I_3 & = \int_t^{\frac{1+t}2}  r^{s+1-d}(r^2-t^2)^{d-s-1}(1-r^2)^{\frac{q-d}2}\,dr, \\
            I_4 & = \int_{\frac{1+t}2}^1 r^{s+1-d}(r^2-t^2)^{d-s-1}(1-r^2)^{\frac{q-d}2}\,dr.
        \end{align*}
        For $q < d$ we have
        \begin{align*}
            |I_3|&\leq (2t)^{d-s-1}(1+t)^{\frac{q-d}2}\left(1-\frac{1+t}2\right)^{\frac{q-d}2}\int_t^{\frac{1+t}2} (r-t)^{d-s-1}\,dr\\
            &\leq C_1 \left(1-\frac{1+t}2\right)^{\frac{q-d}2}\left(\frac{1+t}2-t\right)^{d-s}=C_1\left(\frac{1-t}2\right)^{\frac{q+d}2-s}=C_2\left(1-t\right)^{\frac{q+d}2-s},
        \end{align*}
        where $C_1, C_2>0$, and we used that $t\leq r\leq\frac{1+t}2$ and $ s+1-d > 0 $. An analogous reasoning leads to the same estimate for $q\geq d$.
        Similarly, since $q> d-2$,
        \begin{align*}
        |I_4|&\leq C_3(1-t)^{d-s-1}\int_{\frac{1+t}2}^1 (1-r)^{^{\frac{q-d}2}}\,dr=C_3(1-t)^{d-s-1}\left(1-\frac{1+t}2\right)^{\frac{q-d+2}2}\\
        &=C_4(1-t)^{\frac{q+d}2-s},
        \end{align*}
        where $C_3, C_4>0$, and we used that $r-t\geq \frac{1-t}2$.
        As for the term $I_1$, setting $\tau=t-r$, we obtain
        \begin{align*}
        I_1 &= \int_0^{t-1+2\varepsilon} \frac{(t-\tau)^{d-1}}{t^{2d-s-2} }\tau^{d-s-1}(2t-\tau)^{d-s-1}(1-(t-\tau)^2)^{\frac{q-d}2}\,d\tau\\
        &\leq C_5\int_0^{t-1+2\varepsilon} \tau^{d-s-1}(1-t+\tau)^{\frac{q-d}2}\,d\tau \leq C_6\int_0^{t-1+2\varepsilon} f(t,\tau) \,d\tau,
        \end{align*}
        where $C_5, C_6>0$, and
        $$
        f(t,\tau)=
        \begin{cases}
        \smallskip
        \tau^{d-s-1} \quad & \text {if } \frac{q-d}2\geq 0,\\
        \displaystyle \tau^{\frac{q+d}2-s-1}  \quad & \text {if } -1<\frac{q-d}2< 0.
        \end{cases}
        $$
        Hence, we deduce the estimate
        $$
        |I_1|\leq C_7
        \begin{cases}
        \smallskip
        (t-1+2\varepsilon)^{d-s} \quad & \text {if } \frac{q-d}2\geq 0,\\
        \displaystyle (t-1+2\varepsilon)^{\frac{q+d}2-s}  \quad & \text {if } -1<\frac{q-d}2< 0
        \end{cases}
        $$
        for some $C_7>0$.
        In both cases, if $0<s<\min(d,\frac{q+d}2)$, we can find a positive $\beta$ such that
        $$|I_1|\leq C_7 (t-1+2\varepsilon)^{\beta}\leq C_7 (2\varepsilon)^{\beta}.$$

	 For $t>1$, instead, we write
        $$I=\int_{1-2\varepsilon}^1 \frac{r^{d-1}} {t^{2d-s-2}}  (t^2-r^2)^{d-s-1}(1-r^2)^{\frac{q-d}2}\,dr,$$
        and, assuming $\frac{q+d}2-s>0$, we obtain
       $$|I|\leq C_8\int_{1-2\varepsilon}^1 (1-r)^{\frac{q+d}2-s-1}\,dr=C_8(2\varepsilon)^{\frac{q+d}2-s},$$
       where $C_8>0$.
       
             We conclude that, if
       $d-1<s<\min(d,\frac{q+d}2)$, with $q>d-2$, there exist positive constants $\varepsilon_0\leq 1/4$, $C_9$, and $\beta$ such that
        for any $0<\varepsilon\leq \varepsilon_0$ and any $|1-t|<\varepsilon$, we have
       $$    0\leq  \int_{1-2\varepsilon}^1v_s(t,r)(1-r^2)^{\frac{q-d}2}\,dr\leq C_9\left(|1-t|^{\frac{q+d}2-s}+\varepsilon^{\beta}\right).$$
             Using this estimate it is easy to show that, under these assumptions, the potential is continuous also on $\partial B_1$.

        The case $s=d-1$ follows analogously, by using \eqref{limz2} instead of \eqref{limz3} in the estimate of  ${}_{2}F_{1}$ for $\alpha\sim 1$.          
    \end{proof}

    \begin{remark}\label{regular_oss}
        From Lemma~\ref{regular_lemm}, we infer that $\Ws \ast \mu_{q}^{E}\in C^0(\R^d)$ in particular when $q=s \in [d-2, d)\cap (0,d)$, and when $q=s+2$, with $s\in [d-4,d)\cap (0,d)$. 
        Such a continuity result will be relevant to 
        extend up to the boundary of $E$ the formulas \eqref{formula-qs} and \eqref{formula-qs2} of the following Theorem~\ref{teo:formula_convolution}.
    \end{remark}
    
    We are now in a position to prove the Fourier representation of the potential $\Ws \ast \mu^{E}_{q}$ inside $E$.

    \begin{teo}\label{teo:formula_convolution}
        Let $s \in (0, d)$ and $q\in  [d-2,+\infty)$. Let $\Ws$ be as in \eqref{eq:definition_W}, let $E = RDB_1$ be an ellipsoid, and let $\mu_{q}^{E}$ be the push-forward of the measure $\mu_{q}$ by the map $T^{E}$.
        Then, 
         \begin{equation}\label{inside}
            (\Ws \ast \mu^{E}_{q})(x) = \tilde c_{ d,s, q } \int _{\sph^{d-1}} \frac{ \Psi ( \omega  )  }{ | D R^{ T } \omega  | ^{ s }  }\, {}_{2}F_{1}\left( \frac{s-q}{2}, \frac{ s }{ 2 }; \frac{ 1 }{ 2 } ;\alpha ^{ 2 } ( x, \omega  ) \right) \, d\mathcal{H}^{d-1}(\omega)
        \end{equation}
        for every $x$ in the interior of $E$, 
        where $\Psi=\widehat{\Ws}|_{\sph^{d-1}}$,
        \begin{equation}\label{tildec}
            \tilde c_{d, s, q} = \frac{2^{s - d - 1}\Gamma(1+\frac{q}{2})\Gamma(\frac{s}{2})}{\pi^{d}\Gamma(1+\frac{q-s}{2})} 
        \end{equation}
        and
        \begin{equation}\nonumber
            \alpha(x, \omega) = \frac{x \cdot \omega}{|DR^{T}\omega|} .
        \end{equation}
        In particular, for $q = s$ the potential function $\Ws \ast \mu^{E}_{s}$ is constant in $E$ and is given by
        \begin{equation}\label{formula-qs}
            (\Ws \ast \mu^{E}_{s})(x)=\frac{2^{s-d-2}}{\pi^{d}}s\left(\Gamma\left(\frac{s}{2}\right)\right)^2 \int _{\sph^{d-1}} \frac{ \Psi ( \omega  )  }{ | D R^{ T } \omega  | ^{ s } }\, d\mathcal{H}^{d-1}(\omega)
        \end{equation} 
        for every $x\in E$.
                Finally, for $s\in [d-4,d)\cap (0,d) $ and  $q = s+2$, the potential function $\Ws \ast \mu^{E}_{s+2}$ is, up to an additive constant, a quadratic function in $E$ given by
        \begin{equation}\label{formula-qs2}
            \begin{aligned}
               & (\Ws \ast \mu^{E}_{s+2})(x)\\
               & =\tilde c_{ d,s, s+2 } \int _{ \sph^{d-1}} \frac{ \Psi ( \omega  )  }{ | D R^{ T } \omega  | ^{ s }  }\, d\mathcal{H}^{d-1}(\omega)-
              s\, \tilde c_{ d,s, s+2 }\int _{\sph^{d-1}} \frac{ \Psi ( \omega  )  }{ | D R^{ T } \omega  | ^{ s+2 }  } (x \cdot \omega)^2\, d\mathcal{H}^{d-1}(\omega)
            \end{aligned}
        \end{equation}
	for every $x\in E$.
    \end{teo}

\begin{remark} Theorem~\ref{teo:formula_convolution} provides an alternative way to derive formula (2.31) in \cite{RieszFMMRSV} and in fact extends it to the entire range of $s\in (\max\{d-4,0\},d)$ in any space dimension $d$.
    \end{remark}

    \begin{proof}[Proof of Theorem~\ref{teo:formula_convolution}]
    Formula \eqref{inside} is a consequence of the inversion theorem for the Fourier transform.
    However, we cannot apply the inversion formula directly to $\Ws \ast \mu^{E}_{q}$, since its Fourier transform fails to be integrable at $\infty$ for $2s-q\geq1$. Indeed, by \eqref{hatW}, Lemma~\ref{lemma:fourier_mus}, and due to the asymptotic behavior  \eqref{eq:bessel_infinity} of the tail of the Bessel function at infinity, we have that $\widehat{\Ws\ast \mu _{ q }^{E}}=O(|\xi|^{s-d-\frac q2-\frac12})$ as $|\xi|\to+\infty$. We thus proceed by approximation.

        For $r>0$ let $B_{ r }$ be the ball of radius $r$ centered at the origin and let $\chi_{B_{r}}$ denote its characteristic function.
        We set $\chi_{r} = |B_{r}|^{-1} \chi_{B_{r}}$ and define $P_{ r } = \chi_{ r } \ast ( \Ws\ast \mu _{ q }^{E} )$.
        
        By Lemma~\ref{regular_lemm},
        $P_{r}$ belongs to $L^1_{\loc}(\mathbb{R}^d)$ and
        $P_{ r }$ converges pointwise in $\R^d\setminus \partial E$ (thus almost everywhere in $\R^d$) to $\Ws \ast \mu_{q}^{E}$, as $r\to0^+$.
        If $0<s<\min(d,\frac{q+d}2)$, 
            the function $P_{r}$ is continuous on $\R^d$, and $P_{ r }$ converges pointwise in $\R^d$ to $\Ws \ast \mu_{q}^{E}$, as $r\to0^+$.

        We note that $P_{r}$ is a tempered distribution and $\widehat P_{r} = \widehat \chi_{r} \, \widehat \Ws \, \widehat \mu_{q}^{E}$. To compute $\widehat \chi_{r}$ we follow the computation in \cite[Appendix B.5]{GRAFAKOS2008} and obtain
        \begin{equation}\label{eq:fourier_ball}
            \widehat \chi_{r}(\xi) = \frac{(2\pi)^{\frac{d}{2}}}{|B_{r}||\xi|^{\frac{d}{2}-1}} \int_{0}^{r} \rho^{\frac{d}{2}} J_{\frac{d}{2}-1}(\rho |\xi|) \, d\rho =
            \frac{2^{\frac{d}{2}}\Gamma(1+\frac{d}{2})}{r^{\frac{d}{2}}} \frac{1}{|\xi|^{\frac{d}{2}}} J_{\frac{d}{2}}(r|\xi|),
        \end{equation}
        where we used that $|B_{r}| = r^{d}\pi^{\frac{d}{2}} / \Gamma\big(1+\frac{d}{2}\big)$.
        
        By the homogeneity of $\widehat{\Ws}$, Lemma~\ref{lemma:fourier_mus}, and the asymptotic behavior of Bessel functions in \eqref{eq:bessel_zero}--\eqref{eq:bessel_infinity}, we deduce that $\widehat P_{r}(\xi)=O( |\xi|^{s-d})$
        as $|\xi| \to 0^+$ and $\widehat P_{r}(\xi)= O( |\xi|^{s-\frac{q}{2}-\frac{3}{2}d -1})$ as $|\xi| \to +\infty$. 
        Since $2s-q < d+2$ in our setting, this implies that
        $\widehat P_{r} \in L^{1}(\mathbb{R}^d)$.
        Thus, we can apply the inversion theorem for the Fourier transform to $P_{r}$.
        By Section~\ref{Fourier-Ws}, Lemma~\ref{lemma:fourier_mus}, \eqref{cor:fourier_muEs}, and \eqref{eq:fourier_ball}  we obtain
        \begin{equation}\label{eq:Pr_inverse_fourier}
            P_{ r } ( x ) = \frac{\tilde c_{q, d}}{r^{\frac{d}{2}}}\int _{ \r^d } \frac{ 1 }{ | \xi  | ^{ \frac{ 3 }{ 2 } d - s } | D R^{ T } \xi  | ^{ \frac{ q }{ 2 }  }  } J_{ \frac{ d }{ 2 }  } ( r| \xi  |  ) J_{ \frac{ q }{ 2 }  } ( | D R^{ T } \xi  |  ) \Psi \left( \frac{ \xi  }{ | \xi  |  }  \right) \cos(x \cdot \xi )  \, d\xi ,
        \end{equation}
        where
        \begin{equation}\nonumber
            \tilde c_{q, d} =  \frac{2^{\frac{q-d}{2}}}{\pi^{d}}\Gamma\left(1+\frac{q}{2}\right)\Gamma\left(1+\frac{d}{2}\right) .
        \end{equation}
        In \eqref{eq:Pr_inverse_fourier} the imaginary part can be dropped because $\widehat \chi_{r} \, \widehat \Ws \, \widehat \mu_{q}^{E}$ is even.
        Writing \eqref{eq:Pr_inverse_fourier} in polar coordinates yields
        \begin{equation}\nonumber
            P_{ r } ( x ) = \frac{\tilde c_{ q, d }  }{ r^{ \frac{ d }{ 2 }  }  } \int _{\sph^{d-1} } \frac{ \Psi( \omega  )  }{ | D R^{ T } \omega  | ^{ \frac{ q }{ 2 }  }  } \int _{ 0 } ^{ \infty  } \rho ^{ s - \frac{ q }{ 2 } -\frac{ d }{ 2 } -1 } J_{ \frac{ d }{ 2 }  } ( r\rho  ) J_{ \frac{ q }{ 2 }  } ( \rho | D R^{ T } \omega  |  ) \cos( \rho \omega \cdot x ) \, d\rho \, d\mathcal{H}^{d-1}(\omega) .
        \end{equation}
        Set
        \begin{equation}\label{notation-abt}
        t = | D R^{ T } \omega  | \rho,\quad \beta(r, \omega )  = \frac{ r }{ | D R^{ T } \omega  |  },\quad \text{and}\quad \alpha (x, \omega  )  = \frac{ x\cdot \omega }{ | D R^{ T } \omega  |  } .
        \end{equation}
        Changing variables in the integral with respect to $\rho$ we obtain
        \begin{equation}\nonumber
            P_{ r } ( x ) = \frac{ \tilde c_{ q, d }  }{ r^{ \frac{ d }{ 2 }  }  } \int _{\sph^{d-1}} \frac{ \Psi( \omega  )  }{ | D R^{ T } \omega  | ^{ s - \frac{ d }{ 2 }  }  } I_{ r }( x, \omega  ) \, d\mathcal{H}^{d-1}(\omega),\label{eq:expression_Pr}
        \end{equation}
        where
        \begin{equation}\nonumber
        I_{ r }( x, \omega  ):=\int _{ 0 } ^{ \infty  } t^{ s - \frac{ q }{ 2 } - \frac{ d }{ 2 } -1 } J_{ \frac{ d }{ 2 }  } ( t \beta(r,  \omega )   ) J_{ \frac{ q }{ 2 }  } ( t ) \cos(t \alpha (x, \omega  ) ) \, dt.
        \end{equation}
        To compute $I_r$ we recall that    
        \begin{equation}\nonumber
            J_{-\frac{1}{2}}(z) = \sqrt{\frac{2}{\pi z}} \cos(z)  \qquad \text{ for }z \neq 0 ,
        \end{equation}
        hence
        \begin{equation}\nonumber
        \cos( t\alpha (x, \omega  ) ) = \cos(t|\alpha(x, \omega)|) = \sqrt{ \frac{ \pi   }{ 2 }  }\, t^{ \frac{ 1 }{ 2 }  } |\alpha (x, \omega  )|^{\frac{1}{2}}  J_{ -\frac{ 1 }{ 2 }  } (t |\alpha (x, \omega  )| )  ,
        \end{equation}
        where the last expression on the right-hand side is extended at $t \alpha(x, \omega) = 0$ by continuity.
        Thus, we can rewrite $I_{ r } $ as   
        \begin{equation}\label{rewrite-Ir}
            I_{ r } ( x, \omega  ) = \sqrt{ \frac{ \pi  }{ 2 }  } |\alpha (x, \omega  )|^{\frac{1}{2}} \int _{ 0 } ^{ \infty  } t^{ s - \frac{ q }{ 2 }  - \frac{ d }{ 2 }  - \frac{ 1 }{ 2 }  } J_{ \frac{ d }{ 2 }  } ( t \beta(r, \omega)  ) J_{ \frac{ q }{ 2 }  } ( t ) J_{ -\frac{ 1 }{ 2 }  } (t |\alpha (x, \omega  )| ) \, dt.
        \end{equation}
        For the integral in \eqref{rewrite-Ir} we use the formula
        \begin{align*}
        \int _{ 0 } ^{ \infty  } t^{ \lambda -1 } J_{ \nu  } ( at)&J_{ \mu  } ( bt ) J_{ \rho  } ( ct ) \, dt \\
        & = \frac{ 2^{ \lambda -1 } a^{ \nu  } b^{ \mu  } \Gamma ( \frac{ 1 }{ 2 } ( \lambda +\mu +\nu +\rho  )  )  }{ c^{ \lambda +\mu +\nu  } \Gamma ( \mu +1 ) \Gamma ( \nu +1 ) \Gamma ( 1-\frac{ 1 }{ 2 } ( \lambda +\mu +\nu -\rho  )  )  }  \\
        & \hphantom{=} \times F_4\left(\frac{ 1 }{ 2 } ( \lambda +\mu +\nu -\rho  ) ,\frac{ 1 }{ 2 } ( \lambda +\mu +\nu +\rho  ) ;\mu  + 1, \nu +1; \frac{ b^{ 2 } }{ c^{ 2 }  } ,\frac{ a^{ 2 }  }{ c^{ 2 }  }   \right),
        \end{align*}
        which holds provided that
        \begin{equation}\label{eq:condition_appell}
            \Re ( \lambda +\mu +\nu +\rho  ) > 0 , \quad \Re( \lambda  ) < \frac{ 5 }{ 2 }  , \quad  c > | a | +| b |,
        \end{equation}
        see \cite[Formula 7.1]{BAILEY1936}. 
        For a fixed $x$ in the interior of $E$, $r>0$ and $\omega \in \sph^{d-1}$ we apply the formula with $\lambda = s - \frac{ q }{ 2 } -\frac{ d }{ 2 } +\frac{ 1 }{ 2 } $, $\nu = - \frac{ 1 }{ 2 } $, $\mu = \frac{ d }{ 2 } $,  $\rho = \frac{ q }{ 2 } $, $a = |\alpha (x, \omega  )|$, $b = \beta(r, \omega) $, and $c=1$. 
        Hence, conditions \eqref{eq:condition_appell} translates into $s > 0$, $2s-q < d+4$, and $|\alpha ( x, \omega  )| < 1 - \beta(r, \omega)$.
        The first two conditions are trivially satisfied for $s$ and $q$ in the range under consideration. As for the last one, we observe that $x = p RD \eta $ for some $\eta \in\sph^{d-1}$ and $0 \leq p < 1$, and thus
        \begin{equation}\nonumber
            | \alpha ( x, \omega  ) | = \frac{ | x \cdot  \omega| }{ | D R^{ T } \omega  |  } = \frac{p | \eta \cdot D R^{ T } \omega  |  }{ | D R^{ T } \omega  |  } \leq p < 1 
        \end{equation}
        for every $\omega\in \sph^{d-1}$.
        Since $\beta(r, \omega) \to 0$ as $r \to 0^+$, uniformly with respect to $\omega\in\sph^{d-1}$, there exist $\delta\in(0,1)$ and $r_0>0$ such that for $r<r_0$ we have
        $|\alpha ( x, \omega  )| +\beta(r, \omega)\leq \delta$ for every $\omega\in \sph^{d-1}$. For $x$ in the interior of $E$, $r<r_0$ and $\omega \in \sph^{d-1}$ we can then evaluate the integral in \eqref{rewrite-Ir}, and deduce that 
        \begin{align*}
            I_{ r } ( x, \omega  ) = &\sqrt{ \frac{ \pi  }{ 2 }  } |\alpha ( x,\omega  )|^{\frac{1}{2}} \frac{ 2^{s- \frac{ q }{ 2 } -\frac{ d }{ 2 } -\frac{ 1 }{ 2 }  } \beta^{ \frac{ d }{ 2 }  }(r,  \omega )  |\alpha ( x,\omega  )|^{-\frac{1}{2}} \Gamma \left( \frac{ s }{ 2 } \right)  }{ \Gamma \left( \frac{ d }{ 2 } +1 \right) \Gamma \left( \frac{ 1 }{ 2 }  \right) \Gamma(1+\frac{q-s}{2}) } \\
           & \times F_{ 4 } \left( \frac{s-q}{2},  \frac{ s }{ 2 } ;\frac{ d }{ 2 } +1, \frac{ 1 }{ 2 } ;\beta^{ 2 }(r,  \omega )  ,\alpha ^{ 2 } ( x, \omega  ) \right) \\
           & = \frac{2^{s-\frac{q}{2} - \frac{d}{2} - 1} \beta^{\frac{d}{2}}(r, \omega) \Gamma(\frac{s}{2})}{\Gamma(\frac{d}{2} + 1) \Gamma(1+\frac{q-s}{2})} F_{4}\left( \frac{s-q}{2} , \frac{ s }{ 2 } ;\frac{ d }{ 2 } +1, \frac{ 1 }{ 2 } ;\beta^{ 2 }(r,  \omega )  ,\alpha ^{ 2 } ( x, \omega  ) \right),
        \end{align*}
        where we have used  property (iii) of the Gamma function.
        By \eqref{eq:expression_Pr} and the definition of $\beta$ in \eqref{notation-abt} we conclude that, for $x$ in the interior of $E$ and $ r<r_0$
        \begin{equation}\label{eq:final_Pr}
            P_{r}(x) = \tilde c_{ d,s, q } \int _{ \sph^{d-1}} \frac{ \Psi( \omega  )  }{ | D R^{ T } \omega  | ^{ s }  }F_{4}\left( \frac{s-q}{2}, \frac{ s }{ 2 } ;\frac{ d }{ 2 } +1, \frac{ 1 }{ 2 } ;\beta^{ 2 }(r,  \omega )  ,\alpha ^{ 2 } ( x, \omega  ) \right) \, d\mathcal{H}^{d-1}(\omega)  ,
        \end{equation}
        where $\tilde c_{d, s, q}$ is the constant in \eqref{tildec}.
        
        We now pass to the limit as $r \to 0^+$.
        By \eqref{eq:limit_appel} and by the definition of $\beta$ we have
        \begin{equation}\label{eq:convergence_appel_integrand}
            \lim_{r \to 0} F_{4}\left(\frac{s-q}{2}, \frac{s}{2}; \frac{d}{2} + 1,\frac{1}{2} ; \beta^{2}(r, \omega), \alpha^{2}(x, \omega)\right) ={}_{2} F_{1}\left(\frac{s-q}{2}, \frac{s}{2};\frac{1}{2}; \alpha^{2}(x, \omega)\right) .
        \end{equation}
        Moreover, since $F_{4}$ is analytic in its domain of definition and $|\alpha(x, \omega)| + \beta(r, \omega) \leq \delta < 1$ for $r<r_0$, the convergence \eqref{eq:convergence_appel_integrand} is uniform with respect to 
        $\omega\in\sph^{d-1}$. 
        Hence, passing to the limit in \eqref{eq:final_Pr} and recalling that $P_{r}$ converges to $\Ws \ast \mu_{q}^{E}$ pointwise in the interior of $E$, as $r\to0^+$, we obtain \eqref{inside} for every $x$ in the interior of $E$.

        The statement \eqref{formula-qs} for $q=s$ follows immediately in the interior of $E$ from \eqref{inside}, owing to \eqref{2F10} and property (i) of the Gamma function. Similarly, the statement \eqref{formula-qs2} for $q=s+2$ follows immediately in the interior of $E$ from \eqref{inside}, owing to \eqref{2F1-1}. In fact both formulas hold up to the boundary of $E$ since in both cases the potential $\Ws \ast \mu_{q}^{E}$ is continuous, as observed in Remark~\ref{regular_oss}.
    \end{proof}

    We conclude this section by proving Theorem~\ref{thm:main}.

    \begin{proof}[Proof of Theorem~\ref{thm:main}]
    By Theorem~\ref{thm:exun} the minimizer of $I_{s}^{E}$ exists, is unique, and is the unique measure satisfying the Euler--Lagrange
    conditions~\eqref{eq:third_el_condition}--\eqref{eq:second_el_condition}.
    Therefore, to conclude it is enough to show that $\mu_{s}^{E}$ satisfies
    \eqref{eq:third_el_condition}--\eqref{eq:second_el_condition}.
        Condition \eqref{eq:third_el_condition} is trivially satisfied. Conditions~\eqref{eq:first_el_condition}--\eqref{eq:second_el_condition} 
        follow from Theorem~\ref{teo:formula_convolution} with $q=s$.
    \end{proof}

    \section{Deviation from the radially symmetric case for $s < d-2$}\label{Sec:counter}
    We recall that in Section \ref{sect:superC} we have shown that, for $s\geq d-2$, the minimizer $\mu_s^E$ of $I_s^{E}$ is insensitive to the anisotropy $\Phi$, and in particular if $E=B_1$ we have $\mu_s^{B_1}=\mu_{\iso,s}$. 

    In this section we show that the equality above may fail for $s < d-2$ by providing an explicit example of a kernel $\Ws$  of the form \eqref{eq:definition_W} with continuous and non-negative Fourier transform, for which $\mu_{\iso, d-2}$ is not the minimizer of $I_{s}^{B_{1}}$.

    \begin{lemma}\label{lemma:admissibility_anisotropy}
        Let $ d \geq 3 $ and $s\in(0 ,d-2)$.
        Let $\Ws$ be as in \eqref{eq:definition_W} with $\Phi:{\mathbb S}^{d-1}\to\R$ given by
        \begin{equation}\nonumber
            \Phi(x) = \sum_{i=1}^{d} \alpha _{i}x_{i}^{2} \quad \text{ for }x \in \sph^{d-1},
        \end{equation}
        with $\alpha _{i} > 0$ for $i = 1, \dots, d$.
        Then $\widehat{W_s}$ is non-negative on $\sph^{d-1}$ if and only if
        \begin{equation}\label{cex>0}
            \alpha _{i} \leq \frac{1}{d-s-1} \sum_{j \neq i} \alpha _{j}  \qquad \text{ for every } i = 1, \dots, d .
        \end{equation}
    \end{lemma}

    \begin{proof}
    Let $p_{0}$ and $p_{2,i}$, for $i = 1, \dots, d$, be the homogeneous harmonic polynomials defined as 
    \begin{equation}\nonumber
      p_{0} \equiv1, \quad p_{2,i}(x) = (d-1)x_{i}^{2} - \sum_{j \neq i} x_{j}^{2},
        \end{equation}
        for $x \in\R^d$. Then for $\omega\in \sph^{d-1}$ we write $\Phi $ as 
        \begin{equation}\nonumber
            \Phi(\omega)= \frac{1}{d} \sum_{i=1}^{d}\alpha _{i}p_{0} + \frac{1}{d} \sum_{i=1}^{d}\alpha _{i}p_{2,i},
        \end{equation}
    and hence, for $x\in \R^d$, $x\neq 0$, the kernel $\Ws$ can be rewritten as 
    \begin{equation}\nonumber
        \Ws(x) = \frac{1}{d} \sum_{i=1}^{d} \alpha _{i} \frac{p_{0}}{|x|^{s}} + \frac{1}{d} \sum_{i=1}^{d} \alpha _{i} \frac{p_{2,i}}{|x|^{s+2}}.
    \end{equation}
    From \eqref{Stein-F} we deduce that for $\xi\in \R^d$, $\xi\neq 0$,
        \begin{equation}\label{eq:fourier_quadratic_anis}
            \begin{aligned}
                \widehat \Ws(\xi) & = \frac{1}{d}  2^{d-s}\pi^{\frac{d}{2}} \frac{\Gamma(\frac{d-s}{2})}{\Gamma (\frac{s}{2})} \frac{1}{|\xi|^{d-s}}\sum_{i=1}^{d} \alpha _{i} - \frac{1}{d} \sum_{i=1}^{d} \alpha _{i} 2^{d-s}\pi^{\frac{d}{2}} \frac{\Gamma (1+\frac{d-s}{2})}{\Gamma (1+\frac{s}{2})} \frac{p_{2,i}(\xi)}{|\xi|^{2+d-s}}\\
                                & = \frac{1}{d} 2^{d-s}\pi^{\frac{d}{2}}\frac{\Gamma (\frac{d-s}{2})}{\Gamma (\frac{s}{2})}\frac{1}{|\xi |^{d-s}} \sum_{i=1}^{d}\left(\alpha _{i} - \alpha _{i}\frac{d-s}{s}p_{2,i}\left(\frac{\xi }{|\xi |}\right)\right),
            \end{aligned}
        \end{equation}
        where we have used property (i) of the Gamma function.
    In particular, for $\omega\in\sph^{d-1}$ we have 
        \begin{align*}
            \frac{1}{d}\sum_{i=1}^{d}\left(\alpha _{i} - \alpha _{i}\frac{d-s}{s}p_{2,i}(\omega) \right) & = \frac{1}{d}\sum_{i=1}^{d}\left(\alpha _{i} - \alpha _{i}\frac{d-s}{s}(d\omega_i^2-1) \right) 
            \\
            & = \frac{1}{s}\sum_{i=1}^{d}\alpha_i \sum_{j=1}^d \omega_j^2-\frac{1}{s}\sum_{i=1}^{d}\alpha_i (d-s)\omega_i^2\
            \\
            & =  \frac{1}{s}\sum_{i=1}^{d}\left((1-d+s)\alpha _{i} + \sum_{j \neq i} \alpha _{j}\right)\omega _{i}^{2} .
        \end{align*}
        Thus, \eqref{eq:fourier_quadratic_anis} can be rewritten as 
        \begin{equation}\nonumber
            \widehat \Ws(\xi ) = \frac{1}{|\xi |^{d-s}}2^{d-s}\pi^{\frac{d}{2}}\frac{\Gamma (\frac{d-s}{2})}{\Gamma (\frac{s}{2})}\frac{1}{s} \sum_{i=1}^{d} \left((1-d+s)\alpha _{i} + \sum_{j \neq i} \alpha _{j}\right) \frac{\xi _{i}^{2}}{|\xi|^2} .
        \end{equation}
        Since $1-d+s < 0$ for $s < d-2$, we conclude that $\widehat \Ws \geq 0$ if and only if \eqref{cex>0} is satisfied.
    \end{proof}

    We are now in a position to prove the main result of this section.

    \begin{teo}
        Let $d\geq3$ and $s\in(0,d-2)$. Let $\Ws$ be as in \eqref{eq:definition_W} with $\Phi:{\mathbb S}^{d-1}\to\R$ given by
        \begin{equation}\nonumber
            \Phi(x) = \frac{d-1}{d-s-1} x_{1}^{2} + \sum_{i=2}^{d} x_{i}^{2} = 1 + \frac{s}{d-s-1}x_{1}^{2} \quad \text{ for } x\in \sph^{d-1}.
        \end{equation}
        Then $\Phi$ is continuous, even, and strictly positive, and $\widehat{\Ws}$
        is continuous and non-negative on $\sph^{d-1}$. However, the measure $\mu_{\iso, d-2}$  is not the minimizer of the energy $I_{s}^{B_{1}}$ with kernel~$\Ws$. 
    \end{teo}

    \begin{proof}
    All the properties of $\Phi$ and $\Ws$ are straightforward, except for the sign of $\widehat\Ws$.
     By Lemma~\ref{lemma:admissibility_anisotropy} the non-negativity of $\widehat\Ws$ on $\sph^{d-1}$ is equivalent to
        \begin{equation}\label{eq:admissibility}
            1 \leq \frac{1}{d-s-1} \left(d-2 + \frac{d-1}{d-s-1}\right) .
        \end{equation}
        A simple computation shows that \eqref{eq:admissibility} reduces to $s^{2} -sd \leq 0$, which is true for $0< s < d-2$.

       To prove that  the measure $\mu_{\iso, d-2}$  is not the minimizer of $I_{s}^{B_{1}}$, it is enough to show that
        \begin{equation}\label{eq:failure_EL}
            A := (\Ws \ast \mu_{\iso, d-2} )(e_{2}) - (\Ws \ast  \mu_{\iso, d-2} )(e_{1}) \neq 0.
        \end{equation}
        Indeed, since $ \Ws \ast \mu_{\iso, d-2} \in C^{0}(\r^d)$ by Lemma~\ref{regular_lemm},  equation \eqref{eq:failure_EL} implies that $ \Ws \ast\mu_{\iso, d-2}$ is not constant 
        $\mathcal{H}^{d-1}$-a.e.\ on $\partial B_1$, contradicting \eqref{eq:first_el_condition}.
        
        We note that
        $$
        (\Ws \ast \mu_{\iso, d-2} )(e_{1}) = c_{d-2,d} \int_{\sph^{d-1}}  \frac{1}{|e_{1}-\omega|^{s}}\left(1+\frac{s}{d-s-1}\frac{(1-\omega_{1})^{2}}{|e_{1}-\omega|^{2}}\right) \, d\mathcal{H}^{d-1}(\omega)
        $$
        and by the change of variables $\tilde \omega:=(\omega_2,\omega_1,\omega_3,\dots,\omega_d)$ we can rewrite
        $$
         \begin{aligned}
                (\Ws \ast \mu_{\iso, d-2} )(e_{2}) & = c_{d-2,d} \int_{\sph^{d-1}}  \frac{1}{|e_{2}-\omega|^{s}}\left(1+\frac{s}{d-s-1}\frac{\omega_{1}^{2}}{|e_{2}-\omega|^{2}}\right) \, d\mathcal{H}^{d-1}(\omega) \\
                & =  c_{d-2,d}  \int_{\sph^{d-1}}  \frac{1}{|e_{1}-\tilde\omega|^{s}}\left(1+\frac{s}{d-s-1}\frac{\tilde\omega_{2}^{2}}{|e_{1}-\tilde\omega|^{2}}\right) \, d\mathcal{H}^{d-1}(\tilde\omega).
            \end{aligned}
        $$
        Therefore,
         $$
              A =c_{d-2,d}\frac{s}{d-s-1} \int_{\sph^{d-1}} \frac{\omega_{2}^{2} - (1-\omega_{1})^{2}}{|e_{1} - \omega|^{s+2}} \, d\mathcal{H}^{d-1}(\omega).
         $$
             Since $|e_{1}-\omega|^{2} = 2 - 2\omega_{1}$ for $\omega\in\sph^{d-1}$,
    we have    
    \begin{equation}\nonumber
            |e_{1}-\omega|^{s+2} = 2^{\frac{s+2}{2}} (1-\omega_{1})^{\frac{s+2}{2}} \quad \text{ for } \omega\in\sph^{d-1}.
        \end{equation}
        Hence, the claim \eqref{eq:failure_EL} reduces to showing that
        \begin{equation}\label{eq:integral}
            \int_{\sph^{d-1}} \frac{\omega_{2}^{2} - (1-\omega_{1})^{2}}{(1-\omega_{1})^{\frac{s+2}{2}}}\, d\mathcal{H}^{d-1}(\omega)\neq 0.
        \end{equation}
         
         We consider separately the two cases $d \geq 4$ and $d = 3$.
        We start with the case $ d \geq 4 $. 
        Passing to spherical coordinates, the integral in \eqref{eq:integral} can be rewritten as $ C(d)(I_{1} - I_{2}) $, where $C(d)>0$ is a dimensional constant, and
        \begin{equation}\nonumber 
                \begin{aligned}
                    I_{1} & = \int_{0}^{\pi } \int_{0}^{\pi } \frac{\sin^{d}(\varphi_{1})\cos^{2}(\varphi_{2})\sin^{d-3}(\varphi_{2})}{(1-\cos(\varphi_{1}))^{\frac{s+2}{2}}} \, d \varphi_{1}d\varphi_{2} ,\\
                I_{2} & =  \int_{0}^{\pi } \int_{0}^{\pi } (1- \cos(\varphi_{1}))^{\frac{2-s}{2}}\sin^{d-2}(\varphi_{1})\sin^{d-3}(\varphi_{2}) \, d\varphi_{1} d\varphi_{2}.
            \end{aligned}
        \end{equation}
          To conclude we need to show that $I_{1} - I_{2} \neq 0$.  We write $I_{1}=I_3I_4$, where
         \begin{equation}\nonumber 
                \begin{aligned}
                    I_{3} & = \int_{0}^{\pi } \sin^{d}(\varphi_{1}) (1-\cos(\varphi_{1}))^{\frac{-2-s}{2}} \, d\varphi_{1} ,\\
                I_{4} & = \int_{0}^{\pi } \cos^{2}(\varphi_{2}) \sin^{d-3}(\varphi_{2}) \, d\varphi_{2} .
            \end{aligned}
        \end{equation}
    By the sine and the cosine duplication formula and by \eqref{Beta} we obtain
        \begin{equation}\label{eq:I3}
            \begin{aligned}
                I_{3} & =  2^{d - \frac{s+2}{2}} \int_{0}^{\pi } \sin^{d-2-s}\left(\frac{\varphi_{1}}{2}\right) \cos^{d} \left(\frac{\varphi_{1}}{2}\right)\, d\varphi_{1} \\
                      & = 2^{d +1 - \frac{s+2}{2}} \int_{0}^{\frac{\pi }{2} } \sin^{d-2-s}(\varphi_{1}) \cos^{d} (\varphi_{1})\, d\varphi_{1} = 2^{d-1-\frac{s}{2}} B\left(\frac{d-s-1}{2}, \frac{d+1}{2}\right),
            \end{aligned}
        \end{equation}
        whereas
        \begin{equation}\label{eq:I4}
                I_{4} = 2\int_{0}^{\frac{\pi }{2}} \cos^{2}(\varphi_{2}) \sin^{d-3}(\varphi_{2}) \, d\varphi_{2} = B\left(\frac{d-2}{2}, \frac{3}{2}\right).
               \end{equation}
               Similarly, we write $I_2=I_5I_6$, where
        \begin{equation}\label{eq:I5}
            \begin{aligned}
                I_{5} & = \int_{0}^{\pi }(1 - \cos(\varphi_{1}))^{\frac{2-s}{2}}\sin^{d-2}(\varphi_{1}) \, d\varphi_{1} \\
                      & = 2^{d-\frac{s}{2}} \int_{0}^{\frac{\pi }{2}} \sin^{d-s}(\varphi_{1})\cos^{d-2}(\varphi_{1}) \, d\varphi_{1} = 2^{d-1-\frac{s}{2}} B\left(\frac{d-s+1}{2}, \frac{d-1}{2}\right) 
            \end{aligned}
        \end{equation}
       and
        \begin{equation}\label{eq:I6}
            \begin{aligned}
                I_{6}  & = \int_{0}^{\pi } \sin^{d-3}(\varphi_{2}) \, d\varphi_{2} = 2^{d-2} \int_{0}^{\frac\pi2 } \sin^{d-3}(\varphi_{2}) \cos^{d-3}(\varphi_{2}) \, d\varphi_{2} \\
                &= 2^{d-3} B\left(\frac{d-2}{2}, \frac{d-2}{2}\right) .
            \end{aligned}
        \end{equation}
         Combining \eqref{eq:I3}--\eqref{eq:I6} yields
        \begin{equation}\nonumber 
            \begin{aligned}
               I_1-I_2= I_{3}I_{4} - I_{5}I_{6} & = 2^{d-1-\frac{s}{2}}B\left(\frac{d-s-1}{2}, \frac{d+1}{2}\right) B\left(\frac{d-2}{2}, \frac{3}{2}\right) \\
                                        & \hphantom{=} \,\,- 2^{2d-4-\frac{s}{2}}B\left(\frac{d-s+1}{2}, \frac{d-1}{2}\right)B\left(\frac{d-2}{2}, \frac{d-2}{2}\right) .
            \end{aligned}
        \end{equation}
        By \eqref{Beta0} the previous equality can be rewritten as
        \begin{equation}\nonumber 
            \begin{aligned}
                I_1-I_2= 2^{d-1-\frac{s}{2}}\frac{\Gamma(\frac{d}{2}-1)}{\Gamma (d-\frac{s}{2})} \left( \Gamma \left(\frac{d-s-1}{2}\right)\Gamma \left(\frac{3}{2}\right) - 2^{d-3}\frac{\Gamma (\frac{d-s+1}{2})\Gamma (\frac{d-1}{2})\Gamma (\frac{d}{2}-1)}{\Gamma (d-2)}\right).
            \end{aligned}
        \end{equation}
            By properties (i) and (iii) of the Gamma function we have that
        \begin{align*}
            \Gamma \left(\frac{d-s+1}{2}\right) & = \left(\frac{d-s-1}{2}\right)\, \Gamma \left(\frac{d-s-1}{2}\right) 
        \end{align*}
        and $\Gamma \left(\frac{3}{2}\right) = \frac{1}{2}\Gamma \left(\frac{1}{2}\right) = \frac{1}{2}\sqrt{\pi }$.
        Therefore, $I_1-I_2=0$ if and only if
        \begin{equation}\label{eq:diff4}
            \frac{\sqrt{\pi}}2\Gamma (d-2) - 2^{d-4}(d-s-1)\,\Gamma \left(\frac{d-1}{2}\right)\Gamma \left(\frac{d}{2}-1\right) = 0 .
        \end{equation}
        Finally, properties (ii) of the Gamma function gives
        \begin{equation}\nonumber
            \Gamma \left(\frac{d-1}{2}\right)\Gamma \left(\frac{d}{2}-1\right) = 2^{3-d}\sqrt{\pi } \Gamma (d-2) ,
        \end{equation}
        hence \eqref{eq:diff4} reduces to
        \begin{equation}\nonumber
           \frac{\sqrt{\pi}}2\Gamma (d-2)(2-d+s)=0,
               \end{equation}
               which is never satisfied for $s<d-2$. This concludes the proof for $d \geq 4$.

        When $d = 3$, the previous computations can be repeated with the only difference being that the integration interval for $\varphi_{2}$ is $(0, 2\pi)$, instead of $(0, \pi)$.
        This change simply introduces an extra factor of 2 in the expressions for both $I_{4}$ and $I_{6}$, so the same calculations still lead to the desired conclusion.
    \end{proof}

    A simple continuity argument leads to the following.

    \begin{cor} Let $d\geq 3$ and $s\in(0,d-2)$. For any $\varepsilon>0$, we set
    \begin{equation}\nonumber
    \Phi_{\varepsilon}(x) = \frac{d-1-\varepsilon}{d-s-1} x_{1}^{2} + \sum_{i=2}^{d} x_{i}^{2} \quad \text{ for } x\in \sph^{d-1}
        \end{equation}
        and for $x\in \R^d$, $x\neq 0$,
        \begin{equation}\nonumber
        W_{s,\varepsilon}(x) = \frac{1}{|x|^s} \Phi_{\varepsilon}\bigg(\frac{x}{|x|}\bigg)
    \end{equation}
    while $W_{s,\varepsilon}(0)=+\infty$.

    Then there exists $\varepsilon_0>0$, depending on $s$, such that for any $0<\varepsilon\leq \varepsilon_0$
    the profile $\Phi_{\varepsilon}$ is continuous, even, and strictly positive, 
    $\widehat{W_{s,\varepsilon}}$
    is continuous and strictly positive on $\sph^{d-1}$, but the measure $\mu_{\iso, d-2}$  is not the minimizer of the energy $I_{s}^{B_{1}}$ with kernel $\Ws$ replaced by~$W_{s,\varepsilon}$. 
    \end{cor}

    \bigskip

    \noindent
    \textbf{Acknowledgements.} 
    MGM and EGT are members of GNAMPA--INdAM. 
    MGM and EGT acknowledge support from PRIN 2022 (Project no. 2022J4FYNJ), funded by MUR, Italy, and the European Union -- Next Generation EU, Mission~4 Component~1 CUP~F53D23002760006.
    EGT acknowledges support by the INdAM-GNAMPA project 2025 \lq\lq DISCOVERIES\rq\rq (CUP E5324001950001) and by the Erasmus+ Traineeship program that allowed him to visit LS in Edinburgh.
    LR is supported by the Italian MUR through the PRIN 2022 project n.2022B32J5C, under the National Recovery and Resilience Plan (PNRR), Italy, funded by the European Union  - Next Generation EU, Mission 4 Component 1 CUP~F53D23002710006, and by GNAMPA-INdAM through 2025 projects. 
    LS acknowledges support by the EPSRC under the grants EP/V00204X/1 and EP/V008897/1.

\end{document}